\documentclass{amsart}

\usepackage[applemac]{inputenc}

\usepackage{amsmath, amssymb, amsthm, amsfonts, marginnote}

\input xy
\xyoption{all}

\usepackage{hyperref}

\swapnumbers
\numberwithin{equation}{section}

\theoremstyle{plain}
\newtheorem{theorem}[subsubsection]{Theorem}
\newtheorem{lemma}[subsubsection]{Lemma}
\newtheorem{prop}[subsubsection]{Proposition}

\newtheorem{conj}[subsubsection]{Conjecture}

\theoremstyle{definition}
\newtheorem{defn}[subsubsection]{Definition}

\newtheorem{remark}[subsubsection]{Remark}

\newtheorem{ex}[subsubsection]{Example}

\def\CC{\mathbb{C}}

\def\HH{\mathbb{H}}

\def\PP{\mathbb{P}}

\def\RR{\mathbb{R}}

\def\ZZ{\mathbb{Z}}

\newcommand\cD{\mathcal{D}}
\newcommand\cE{\mathcal{E}}
\newcommand\cF{\mathcal{F}}

\newcommand\cN{\mathcal{N}}

\newcommand\frA{\mathfrak{A}}
\newcommand\frB{\mathfrak{B}}

\newcommand\frD{\mathfrak{D}}

\newcommand\frN{\mathfrak{N}}

\newcommand\frP{\mathfrak{P}}

\newcommand\frV{\mathfrak{V}}

\newcommand\frX{\mathfrak{X}}
\newcommand\frY{\mathfrak{Y}}

\newcommand\frg{\mathfrak{g}}

\newcommand\frr{\mathfrak{r}}

\newcommand\frz{\mathfrak{z}}

\newcommand{\Bun}{\textup{Bun}}

\renewcommand{\Im}{\textup{Im}}
\newcommand{\Ind}{\textup{Ind}}

\newcommand\Loc{\textup{Loc}}

\newcommand{\Perf}{\textup{Perf}}

\newcommand\Rep{\textup{Rep}}

\newcommand\Aut{\textup{Aut}}
\newcommand\Hom{\textup{Hom}}

\newcommand\PGL{\textup{PGL}}

\newcommand\SL{\textup{SL}}

\newcommand{\wt}[1]{\widetilde{#1}}
\newcommand{\wh}[1]{\widehat{#1}}
\newcommand\quash[1]{}

\newcommand{\Sh}{\mathit{Sh}}

\newcommand{\Coh}{\textup{Coh}}

\newcommand{\beq}{\begin{equation}}
\newcommand{\eeq}{\end{equation}}

\newcommand{\ssupp}{\mathit{ss}}

\newcommand{\oo}{\infty}

\newcommand{\ol}{\overline}

\newcommand{\Fun}{\operatorname{Fun}}

\newcommand{\Flags}{\operatorname{Flags}}
\newcommand{\Subsets}{\operatorname{Subsets}}

\newcommand{\dgCat}{\operatorname{dgCat}}

\title{A microlocal criterion for commuting nearby cycles}
\author{David Nadler}
\address{Department of Mathematics\\University of California, Berkeley\\Berkeley, CA  94720-3840}
\email{nadler@math.berkeley.edu}
\dedicatory{}
\date{\today}

\begin{document}


\maketitle

\begin{abstract}
We present a microlocal criterion 
for the equivalence of iterated nearby cycles along different flags of subspaces
in a higher-dimensional base. 
The result is motivated by its application to Hitchin systems in the context of Betti geometric Langlands~\cite{NYverlinde}.

\vspace{1em}

{\it I thank Claude Sabbah for bringing papers of Maisonobe \cite{M} and Kocherperger~\cite{Ko, Ko2} to my attention. Among many interesting results, they include assertions implying the main assertions of this paper. I will leave this paper on the arXiv as a possibly useful additional reference.}  
\end{abstract}

\tableofcontents

%
%
%
%
%
%
%

\section{Introduction}

The study of nearby cycles over higher-dimensional bases was pioneered by
Deligne, Gabber, Illusie, Laumon, Orgogozo
~\cite{ILO,  L, O}. 
We recommend the recent papers of Illusie~\cite{Ill} and Lu-Zheng~\cite{LZ}
 for an overview and some of the latest advances. We have also benefited from the user-friendly
 account in Illusie's notes~\cite{I}.
 
The current paper is closest in spirit and results to those of  L\^e~\cite{Le} and Sabbah~\cite{S} who work in an analytic setting with topological methods. 
We adopt this approach, and thus bypass  intricacies specific to the algebraic setting. 
Our main result provide a generalization of the result of L\^e~\cite{Le} to situations with microlocal bounds  but not necessarily stratifications.
Our specific arguments are closest to those appearing in joint work with Shende~\cite{NS}.
The analytic setting is also sufficient  for the motivating application to Hitchin systems found in \cite{NYverlinde} and  discussed below in \S\ref{s:apps}. 


Given the well-established theory of nearby cycles over curves, the  primary question
over higher-dimensional bases 
  is how nearby cycles vary along different curves in the base. 
The main result of this paper  provides a microlocal criterion, in terms of the singular support of Kashiwara-Schapira~\cite{KS}, for the equivalence of iterated nearby cycles along different flags of subspaces.
The intuition behind the criterion is simple if one views  sheaves as Lagrangian branes: a family of Lagrangian branes in a fixed symplectic manifold satisfies the criterion if its closure at the central limit parameter is again Lagrangian.

In the rest of the introduction, we first discuss a model situation of our main result and the notions that go into it. We then  sketch our motivating application \cite{NYverlinde} to  geometric Langlands~\cite{BD}, specifically towards an automorphic  ``Verlinde formula" in the Betti version of the theory~\cite{BN}.

\subsection{Main result}\label{s:intro main result} 

We will state our main result here in a model situation. (For a general formulation, see Theorem~\ref{thm: main}.) 
We work throughout in a ``tame"  setting such as subanalytic sets~\cite{BM}, or more general o-minimal structures~\cite{vdD, vdDM}.  

Consider a complex manifold $X$ equipped with two holomorphic functions $f_1, f_2$.  We will only be interested in $X$ in a neighborhood of the simultaneous zeros of $f_1, f_2$,  so 
will assume $f_1, f_2$ take values in the open unit disk $D = \{z\in \CC \, |\,  |z|^2 < 1 \}$.

Consider the product map
$$
\xymatrix{
f = f_1 \times f_2:X\ar[r] &  S =  D^2
}$$ 
 the open locus $X^\times  = f^{-1}( \{f_1, f_2  \not = 0\}$, and the special fiber
 $X_0 = \{f_1, f_2 = 0\}$.
 

Given  a weakly constructible complex of sheaves $\cF\in \Sh(X^\times)$, we construct (see Proposition~\ref{prop: lax}) a natural diagram of complexes 
on the special fiber
\beq\label{eq:intro bc}
\xymatrix{
 \psi_{f_1} \psi_{f_2}\cF & \ar[l] \Psi\cF \ar[r] & \psi_{f_2} \psi_{f_1}\cF
}
\eeq
Here $ \psi_{f_1}, \psi_{f_2}$ are the traditional nearby cycles functors for $f_1, f_2$, which we apply in iteration, and  $\Psi$ is an ``unbiased" nearby cycles functor involving only the geometry of the subset $X^\times \cup X_0 \subset X$ (see Definition~\ref{def:iterated nc}). Moreover, the objects of \eqref{eq:intro bc} come equipped with natural monodromy $\ZZ^2$-actions 
which the maps of \eqref{eq:intro bc} intertwine.

Our main result gives a sufficient condition for the maps of \eqref{eq:intro bc} to be equivalences,
and hence for the traditional nearby cycles $ \psi_{f_1}, \psi_{f_2}$ to ``commute". 
It applies immediately to the situation when $f:X\to S$ is a submersion. If this is not the case, we can always arrive at this case via the  graph construction:
we replace the map $f:X\to S$ with the projection $\pi: X \times S\to S$, consider the graph $\gamma:X \hookrightarrow X \times S$, $\gamma(x) = (x, f(x))$, 
and work with complexes supported on the graph $\Gamma_f = \gamma(X)$.

Recall that to a complex of sheaves $\cF\in \Sh(X)$, following Kashiwara-Schapira~\cite{KS}, one can assign its singular support $\ssupp(\cF) \subset T^*X$. The singular support $\ssupp(\cF)$ is closed and conic, and additionally, Lagrangian if and  only if $\cF$ is weakly constructible. In particular, the intersection of  $\ssupp(\cF)$  with the zero-section is the traditional support of $\cF$.

Here is our main result specialized to the given setup. We explain the hypotheses in the discussion immediately following.
  
\begin{theorem}\label{thm:intro}
Assume $f = f_1 \times f_2:X\to S$ is a submersion.

Let $\Lambda \subset T^*X^\times$ be a closed conic Lagrangian.

Suppose  $\Lambda$ is (i)  $f$-non-characteristic and (ii) $f$-Thom at the origin.

Let $\cF\in \Sh(X^\times)$ be a complex of sheaves with singular support $\ssupp(\cF) \subset T^* X^\times$ contained within $\Lambda$.

Then the natural maps are equivalences
$$
\xymatrix{
 \psi_{f_1} \psi_{f_2}\cF & \ar[l]_-\sim\Psi\cF \ar[r]^-\sim & \psi_{f_2} \psi_{f_1}\cF
}
$$
\end{theorem}

\begin{remark}
Theorem~\ref{thm:intro} is a particular case of the compatibility of $n$-fold iterated nearby cycles  proved
 in Theorem~\ref{thm: main}.
 To briefly convey its content, for a map of complex manifolds $f = f_1 \times \cdots \times f_n:X\to D^n$,  set  $X^\times = f^{-1}((D^\times)^n)$, $X_0 = f^{-1}(0)$. Given 
  a sheaf $\cF\in \Sh(X^\times)$
satisfying hypotheses as in Theorem~\ref{thm:intro}, and a permutation $\sigma \in \Sigma_n$, 
 we construct canonical equivalences 
 $$
\xymatrix{
 \psi_{f_1} \psi_{f_2}\cdots \psi_{f_n}\cF  \simeq \psi_{f_{\sigma(1)}} \psi_{f_{\sigma(2)}}\cdots \psi_{f_{\sigma(n)}}\cF
}
$$
More precisely, inside of the $\oo$-category of functors $\Fun( \Sh(X^\times), \Sh(X_0))$, we construct a contractible $\oo$-groupoid whose objects include the  possible iterations
$\psi_{f_{\sigma(1)}} \psi_{f_{\sigma(2)}}\cdots \psi_{f_{\sigma(n)}}\cF$, for $\sigma\in \Sigma_n$.   
\end{remark}

To explain the hypotheses of Theorem~\ref{thm:intro}, let us briefly  recall some microlocal constructions.
 To a holomorphic map of complex manifolds  $f:X\to S$, we have the Lagrangian  correspondence of cotangent bundles
$$
\xymatrix{
T^*S & \ar[l]_-p f^*(T^*S)  \ar[r]^-{df^*} & T^*X
}
$$
where $f^*(T^*S) = T^*S \times_{S} X$ is the pullback bundle, 
 $p$ is the evident projection, and $df^*$ is the pullback of covectors. If $f$ is a submersion, then $df^*$ is injective, and we have a short exact exact sequence of vector bundles
\beq\label{eq:ses}
\xymatrix{
0 \ar[r] & f^*(T^*S)  \ar[r]^-{df^*} & T^*X \ar[r]^-\Pi  & T^*_f \ar[r] & 0
}
\eeq
where $T^*_f$ is the relative cotangent bundle. Note for any $s\in S$, with fiber $X_s = f^{-1}(s)$, we have a canonical identification $T^*_f|_{X_s}  \simeq T^*X_s$. 
%

The  hypotheses (i) and (ii)  of Theorem~\ref{thm:intro} will place a constraint on the interaction
of  the subset $\Lambda \subset T^*X$ with  the kernel and cokernel of the short exact sequence~\eqref{eq:ses} respectively.

\begin{defn} Let $\Lambda \subset T^*X$ be a subset, and 
 $f:X\to S$ a submersion.

We say that $\Lambda$ is {\em $f$-non-characteristic} if the intersection 
$$ \Lambda \cap df^*(f^*(T^*S))
$$ lies in the zero-section of $T^*X$.

\end{defn}

\begin{remark}
For $Y \subset X$ a submanifold, the conormal bundle $T^*_Y X \subset T^*X$ is $f$-non-characteristic  if and only if 
the restriction $f|_Y:Y \to S$ is a submersion.
\end{remark}


Given  a submersion $f:X\to S$, and a  subset $\Lambda \subset T^*X$, we refer to its image $\Lambda_f = \Pi(\Lambda) \subset T^*_f$ in the relative cotangent bundle as  the {\em $f$-projection} of $\Lambda$. 
 We will be interested in the closure
of the $f$-projection
$$
 \ol\Lambda_f \subset T^*_f
 $$  
 in particular
 its intersection with   fibers $X_s$, for $s\in S$, which we denote by
 $$
\xymatrix{
 \ol \Lambda_{f, s} = \ol\Lambda_f \times_S {X_s} \subset T^*_f|_{X_s} \simeq T^* X_s
 }
 $$

\begin{remark}
 Note $ \ol\Lambda_f = \ol{\Pi(\Lambda)}$ contains $\Pi(\ol \Lambda)$, but in general they are not equal. For example, take 
 the projection 
 $f:X = \CC^2\to \CC = S$, $f(x, y) = x$, the punctured parabola $Y = \{ x \not = 0, x = y^2\}$, and its conormal bundle $\Lambda = T^*_Y X \subset T^*X$. Then $\ol \Lambda_{f, 0}$ is the cotangent fiber $T^*_0 X_0$, while  $\Pi(\ol \Lambda)|_{X_0}$ consists of only the zero covector in $T^*_0 X_0$.
 \end{remark}

\begin{defn}\label{def:intro thom} Let $\Lambda \subset T^*X$ be a conic Lagrangian, and
 $f:X\to S$ a submersion.

We say that $\Lambda$ is {\em $f$-Thom} at a point $s\in S$ if the fiber $\ol\Lambda_{f, s}  \subset T^*X_s$ of the closure of the $f$-projection of $\Lambda$ is isotropic.

\end{defn}

\begin{remark}
We explain in Proposition~\ref{p:thom=thom}  how the Thom condition 
for Lagrangians generalizes the Thom $A_f$ condition for stratified maps. 
\end{remark}

\begin{remark}
When the base $S$ is a curve,  any  conic Lagrangian $\Lambda \subset T^*X$ is $f$-Thom at each point of $S$. See Proposition~\ref{p:dim 1}.
\end{remark}

%
%
Before turning to a motivating application in the next section, let us illustrate the content of the theorem with the following canonical non-example.

\begin{ex}
Consider the projection $f:X = \CC\PP^1 \times \CC^2 \to \CC^2 = S$.

 Let $Y \subset X$ be the blow up locus of pairs $\ell \in \CC\PP^1$, $v\in \CC^2$ with $v\in \ell$. 
 
 Let $k_Y \in \Sh( X)$ be the constant sheaf on $Y$. Then $\psi_1 \psi_2 k_Y  \simeq k_{[1, 0]}$,
 $\psi_2 \psi_1 k_Y  \simeq k_{[0, 1]}$ are skyscraper sheaves at the respective points $[1,0]$, $[0,1]$ of the special fiber $X_0 = \CC\PP^1$, so not isomorphic.
 
 For $\Lambda = T^*_Y X$ the conormal bundle, one finds $\ol \Lambda_{f, 0} = T^*X_0$ is the entire cotangent bundle of the special fiber, so not isotropic.

\end{ex}


\subsection{Application}\label{s:apps} We sketch here the motivating application of this paper found in joint work with Zhiwei Yun~\cite{NYverlinde}. It is part of a broad undertaking, developed also with David Ben-Zvi~\cite{BN}, to
understand the Betti version of Geometric Langlands. The reader familiar with the subject could skip to the highlighted assertions at the end of the section. 

Fix a complex reductive group $G$, with Langlands dual $G^\vee$, and a smooth complex projective curve $C$. 

In Betti Geometric Langlands, one studies the moduli $\Bun_G(C)$ of $G$-bundles on $C$, and in particular the Hitchin system on its cotangent bundle $T^*\Bun_G(C)$.  The special fiber of the Hitchen sysem is the global nilpotent cone $\cN \subset  T^*\Bun_G(C)$
of Higgs bundle $(\cE, \varphi)$ whose Higgs field $\varphi$ is everywhere nilpotent. 

The overarching challenge is to understand the automorphic category $\Sh_\cN( \Bun_G(C))$ of complexes of sheaves on $\Bun_G(C)$ with singular support in $\cN$. At each closed point $c\in C$, 
spherical Hecke operators act on  $\Sh_\cN( \Bun_G(C))$, and thus give an action of 
the tensor category $\Rep(G^\vee)$ via the geometric Satake correspondence~\cite{MV}.
One can integrate these local actions to obtain a global action of the tensor category $\Perf(\Loc_{G^\vee}(C))$ of perfect complexes on the moduli $\Loc_{G^\vee}(C)$ of $G^\vee$-local systems on $C$~\cite{NY}. 

In its most basic form, the Betti Geometric Langlands conjecture asserts an equivalence of $\Perf(\Loc_{G^\vee}(C))$-module categories
$$
\xymatrix{
\Ind\Coh_\cN(\Loc_{G^\vee}(C)) \ar[r]^-{\sim ?}  & \Sh_\cN( \Bun_G(C))
}
$$
where $\Ind\Coh_\cN(\Loc_{G^\vee}(C))$ denotes ind-coherent sheaves on $\Loc_{G^\vee}(C)$ with nilpotent singular support~\cite{AG}. 

Among many possible variations and generalizations, it is worth highlighting that one may consider parahoric level-structure at a finite subset of marked points $S\subset C$. In particular, one may take Iwahori-level structure, and study the automorphic moduli $\Bun_G(C, S)$ of $G$-bundles on $C$ with a $B$-reduction along $S$, along with  the corresponding spectral moduli $\Loc_{G^\vee}(C, S)$ of $G^\vee$-local systems on $C\setminus S$ with a  $B^\vee$-reduction near $S$.
 With this setup, the Betti Geometric Langlands conjecture asserts an equivalence of $\Perf(\Loc_{G^\vee}(C, S))$-module categories
$$
\xymatrix{
\Ind\Coh_\cN(\Loc_{G^\vee}(C, S)) \ar[r]^-{\sim ?}  & \Sh_\cN( \Bun_G(C, S))
}
$$

\begin{remark}
Note that the spectral category $\Ind\Coh_\cN(\Loc_{G^\vee}(C, S))$ only depends on the pair $S\subset C$ through their  topology not the complex structure on $C$. Thus implicit in the above conjecture is the subsidiary conjecture that the automorphic category  $\Sh_\cN( \Bun_G(C, S))$ is likewise a topological invariant of the pair $S\subset C$. 
In the informal discussion to follow, we will proceed assuming this, and only mention here that it is known in some concrete but nontrivial situations such as for curves of genus one~\cite{LN}.
\end{remark}

In addition to its topological invariance, the spectral category $\Ind\Coh_\cN(\Loc_{G^\vee}(C, S))$  enjoys many structures from topological field theory, notably a categorical  analogue of the Verlinde formula.
Namely, if one introduces a pair-of-pants decomposition of $C$, one can recover $\Ind\Coh_\cN(\Loc_{G^\vee}(C, S))$ from the spectral categories $\Ind\Coh_\cN(\Loc_{G^\vee}(C_i , S_i))$ of the pairs $S_i \subset C_i$ arising in the pair-of-pants decomposition~\cite{BNglue}. In particular, one can recover all spectral invariants from the case of genus one curves with at most three marked points.

It is here  that the results of this paper enter the story, specifically
when we seek a parallel   Verlinde formula for the automorphic category
$ \Sh_\cN( \Bun_G(C, S))$.
 Unfortunately, since the moduli $ \Bun_G(C, S)$ depends essentially on the complex structure on $C$, 
it is not possible to directly apply a pair-of-pants decomposition of $C$. Instead, one may consider a degeneration  to a nodal curve $C\rightsquigarrow C_0$, with marked points $S \rightsquigarrow S_0$ disjoint from the nodes, and the corresponding degeneration of moduli of bundles
\beq\label{eq:degen}
\xymatrix{
 \Bun_G(C, S) \ar@{~>}[r]  & \Bun_G(C_0, S_0)
}
\eeq

To say more specifically what we mean by the moduli $ \Bun_G(C_0, S_0)$ for the nodal curve,   let us briefly introduce some further notation.  Write  $R \subset C_0$ for the set of nodes, introduce  the normalization $p:\tilde C_0\to C_0$, and consider the pre-images
$\wt S_0 = p^{-1}(S_0) $, $R_+ \coprod R_- = p^{-1}(R)$ as finite subsets of $\tilde C_0$.  Here we have a canonical bijection $R_+ \simeq R_-$ identifying pairs of points projecting under $p$ to the same node. 
By definition
$ \Bun_G(C_0, S_0)$  classifies a $G$-bundle on $\tilde C_0$, a $B$-reduction along  $\tilde S_0 \coprod R_+ \coprod R_-$,  and an isomorphism of the induced $H$-bundles over $R_+, R_-$ covering the bijection $R_+ \simeq R_-$.

For more discussion on the structure of the  degeneration~\eqref{eq:degen}, we recommend~\cite{faltings, faltings verlinde, solis} where it is realized as a global version of the Vinberg semigroup for the loop group.

Now to relate the automorphic categories under the degeneration~\eqref{eq:degen}, one may take nearby cycles
\beq\label{eq:nearbybun}
\xymatrix{
 \psi:\Sh_\cN(\Bun_G(C, S)) \ar[r]  & \Sh_\cN(\Bun_G(C_0, S_0))
}
\eeq
and attempt to identify the ``image" of $\psi$. Following the gluing paradigm for spectral categories~\cite{BNglue}, one arrives at the following conjecture for automorphic categories. Its resolution would provide the sought-after automorphic Verlinde formula.

\begin{conj}\label{conj:intro}
The left adjoint $\psi^L$ to the nearby cycles \eqref{eq:nearbybun} is the universal functor  co-equalizing the pair of  affine Hecke actions 
on $\Sh_\cN(\Bun_G(C_0, S_0))$ at each node of $C_0$.

\end{conj}

Via Radon transforms, 
one can identify the affine Hecke actions of the conjecture with adjoints to nearby cycles in further degenerations to nodal curves with chains of projective lines.
The main result of the current paper provides the requisite compatibilities between these nearby cycles. It thus allows us to deduce the co-equalization (but not the universality) asserted in Conjecture~\ref{conj:intro}.

\begin{theorem}[\cite{NYverlinde}]\label{thm:ny}
The left adjoint $\psi^L$ to the nearby cycles \eqref{eq:nearbybun} co-equalizes the  pair of affine Hecke actions 
on $\Sh_\cN(\Bun_G(C_0, S_0))$ at each node of $C_0$.

\end{theorem} 

As a consequence of the theorem and the  gluing of spectral categories, we obtain 
 that any map of Hecke module categories 
\beq\label{eq:given}
\xymatrix{
\Ind\Coh_\cN(\Loc_{G^\vee}(C_0, S_0)) \ar[r]  & \Sh_\cN( \Bun_G(C_0, S_0))
}
\eeq
induces a map of Hecke module categories 
\beq\label{eq:obtained}
\xymatrix{
\Ind\Coh_\cN(\Loc_{G^\vee}(C, S)) \ar[r]  & \Sh_\cN( \Bun_G(C, S))
}
\eeq
Moreover, one can check the construction enjoys many essential properties, notably, if \eqref{eq:given} is compatible with parabolic induction, then so is \eqref{eq:obtained}. (In contrast, if one constructs a functor of the form  \eqref{eq:obtained} via the spectral action of
$\Perf(\Loc_{G^\vee}(C))$, as for example done in~\cite{NY3pts}, it is difficult to establish its compatibility with parabolic induction.)

Of course, of particular interest are situations when we can take \eqref{eq:given} to be a known equivalence and then  check \eqref{eq:obtained} is also an equivalence. Let us mention two such situations to appear in future work.

In one direction,  we can take $C = E$ an elliptic curve with no marked points, and $C_0 = \PP^1/\{0\sim \oo\}$ a nodal cubic. Then via Radon transforms, one can identify $ \Sh_\cN( \Bun_G(C_0, S_0))$ with the regular bimodule for the affine Hecke category.
Thus Theorem~\ref{thm:ny} provides a  map from the  cocenter of the affine Hecke category to the automorphic Betti Geometric Langlands in genus one. 
Following Bezrukavnikov's local tamely ramified 
geometric Langlands equivalence~\cite{B}, one can construct an equivalence  \eqref{eq:given}. We expect it is then possible to confirm~\eqref{eq:obtained} is likewise an equivalence, thus establishing the 
Betti Geometric Langlands equivalence in genus one.

In another direction, we can take $C = \PP^1$ genus zero with four marked points, and  $C_0 = \PP^1 \vee \PP^1$ a nodal pair of genus zero curves each with two additional marked points. Following~\cite{NY3pts}, when $G$ is rank one, one can construct an equivalence~\eqref{eq:given}. We expect it is then possible to confirm~\eqref{eq:obtained}   is 
likewise an equivalence, thus establishing the  Betti Geometric Langlands equivalence in the case of rank one groups over $\PP^1$ with four marked points.

%
%
%
%
%
%

\subsection{Acknowledgements}
This paper grew out of discussions with Zhiwei Yun devoted to the applications described in Sect.~\ref{s:apps}.
His generous interest and comments were pivotal to its development. I am  grateful to David Ben-Zvi and Vivek Shende for many  discussions on related topics.  I am also grateful to Misha Grinberg and Kari Vilonen for their insightful comments recorded in Remark~\ref{rem:gv}.

This work was supported by NSF grant DMS-1802373. It was also supported by 
NSF grant DMS-1440140 while the author was  in residence at MSRI during the Spring 2020 semester.

\section{Preliminaries}

\subsection{Tame geometry}

We work throughout in the ``tame" setting of subanalytic sets~\cite{BM} or more general o-minimal structures~\cite{vdD, vdDM}.
We list here the key  properties we will use.

All manifolds are assumed to be real analytic.  All subsets are assumed to be subanalytic, or more generally definable in an o-minimal structure. All maps are assumed to have graphs that are such subsets.

\subsubsection{Whitney stratifications}

Let $M, N$ be manifolds.

Let $X \subset M$ be a closed subset, and $p>0$ a positive integer.

For a locally finite collection of subsets $\frA$ of $X$ there is a $C^p$ Whitney stratification $\frX$ of $X$
such that each $A\in \frA$ is a union of strata. In this situation,
one says  that $\frX$ is compatible with $\frA$.

For a proper map $f: X \to N$, and locally finite collections of subsets $\frA$ of $M$ and $\frB$ of $N$, 
there are $C^p$ Whitney stratifications $\frX$ of $X$ and  $\frN$ of $N$  compatible  with $\frA$ and $\frB$
 such that 
 for each  
 $S\in \frX$ 
 we have $f(S)\in \frN$, 
 and  the restriction 
 $f|_S:S \to f(S)$ is a $C^p$ submersion. 
In this situation, one says  $\frX$, $\frN$ are compatible with  $f$.
 
%

\subsubsection{Curve selection}
Let $M$ be a manifold.

For a subset $A \subset M$, and a point $x\in \ol A \setminus A$, there exists a map $\gamma:[0, 1) \to M$ such that
$\gamma(0) = x$, and $\gamma(0, 1) \subset A$.

\subsubsection{Transversality}

\begin{lemma}\label{l:trans}
Let $M$ be a manifold, and $f: M\to \RR$ a non-constant function.

Let $\frX$ be a finite collection of relatively compact submanifolds of $M$.

There exists $\epsilon>0$ such that over the open $f^{-1}(0, \epsilon) \subset M$,  the restriction $df|_{X}$ to each $X\in \frX$ is non-zero.
\end{lemma}

\begin{proof}

It suffices to prove the assertion for a single relatively compact  submanifold $X\subset M$.
Consider the subset $A = \{x\in X \, |\, f(x)> 0, (df|_X)_x = 0\}$. If the assertion does not hold, then there exists a sequence $x_n \in A$ such that $\lim_{n\to \oo} f(x_n) = 0$. By the compactness of $\ol X \subset M$, after passing to a subsequence, there is a limit $\lim_{n\to \oo} x_n = x_\oo \in \ol X$. By curve selection applied to $A$, 
there exists a map $\gamma:[0, 1) \to M$ such that
$\gamma(0) = x_\oo$, and $\gamma(0, 1) \subset A$. Consider the restriction $f|_\gamma:[0,1)\to \RR$, and note $f|_\gamma(0) = f(x_\oo) = 0$, and $f|_\gamma (t) > 0 $, for $t\in (0,1)$. Thus there exists $t\in (0, 1)$ such that $(f|_\gamma)'(t)  > 0$. But  this contradicts $(df|_X)_{\gamma(t)} = 0$. 
\end{proof}


\subsection{Microlocal geometry}

We record here notation for various standard constructions about manifolds and their cotangent bundles.

 Given a manifold $X$, we write $T^*X$ (resp. $TX$) for its cotangent (resp. tangent) bundle, and denote its points by pairs $(x, \xi_x) \in T^*X$ where $x\in X$, $\xi_x \in T^*_x X$ (resp. $(x, v_x)\in TX$ where $x\in X$, $v_x \in T_x X$). 
 Given a submanifold $Y \subset X$, we write $T^*_Y X$ (resp. $N_Y X$) for its conormal  (resp. normal) bundle.

Given a submersion $\pi:X \to Y$, we have the dual   exact sequences of vector bundles
 $$
 \xymatrix{
 0 \ar[r] &  \pi^*(T^*Y)  \ar[r]^-{(d\pi)^*} & T^*X \ar[r]^-\Pi & T^*_\pi \ar[r] & 0
 }
 $$
 $$
 \xymatrix{
 0 & \ar[l]   \pi^*(TY) &  \ar[l]_-{ d\pi}  TX  &  \ar[l]_-{\Pi^*}  T_\pi &  \ar[l]  0
 }
 $$
We refer to 
$T^*_\pi$ (resp. $T_\pi$) as the relative cotangent (resp. tangent) bundle.

We equip $T^*X$ with the standard symplectic structure $\omega = d\xi dx$ with primitive $\xi dx$ and Liouville vector field $\xi\partial_{\xi} = \omega^{-1}(\xi dx)$. We say a subset $Y \subset T^*X$ is {\em conic} if it is invariant under the dilations $e^{t \xi \partial_\xi} \cdot (x, \xi_x) = (x, e^t \xi_x)$, for $t\in \RR$. We say a subset $Y \subset T^*X$ is {\em isotropic} if for all submanifolds $Z \subset Y$, we have $\omega|_Z = 0$ (equivalently, there exists a dense locally closed submanifold $Z \subset Y$ of possibly varying dimension such that 
$\omega|_Z  = 0$).
We say a subset $Y \subset T^*X$ is {\em Lagrangian} if it is isotropic and of pure dimension $\dim Y = \dim X$. For any submanifold $Y \subset X$, its conormal bundle $T^*_Y X \subset T^*X$ is conic Lagrangian. In particular,  the conormal bundle of $X$ itself is the zero-section $T^*_X X \simeq X$.
 
 We write $T^\oo X$ for the  cosphere bundle
 $$
 \xymatrix{
 T^\oo X = (T^*X \setminus T^*_X X)/\RR_{>0}
 }
 $$
 We equip $T^\oo X$ with the standard co-oriented contact structure $\zeta =\ker(\xi dx)$.
 We say a subset $Y \subset T^\oo X$ is {\em isotropic} if for all submanifolds $Z\subset Y$, we have $TZ \subset \zeta$ (equivalently, there exists a dense locally closed submanifold $Z \subset Y$ of possibly varying dimension such that 
$TZ   \subset \zeta$). We say a subset $Y \subset T^*X$ is {\em Legendrian} if it is isotropic and of pure dimension $\dim Y = \dim X-1$.

 For conic $\Lambda \subset T^*X$, we write $\Lambda^\oo \subset T^\oo X$ for its projectivization
 $$
 \xymatrix{
 \Lambda^\oo  = (\Lambda \setminus (\Lambda \cap T^*_X X))/\RR_{>0}
 }
 $$
 Note $\Lambda$ is conic Lagrangian if and only if $\Lambda^\oo$ is Legendrian.

 \begin{lemma}
 Let $\Lambda^\oo \subset T^\oo X$ be closed and isotropic. 
 
 Then there exists a Whitney stratification $\{ X_i\}_{i \in I}$ of $X$ such that we have the containment
 $$
 \Lambda^\oo \subset  \bigcup_{i \in I} T^\oo_{X_i} X
 $$ \end{lemma}
 
 \begin{proof}
 Consider the natural projection $\pi:T^\oo X \to X$.   Choose Whitney stratifications of $\Lambda^\oo$, $X$  compatible with $\pi$. Thus for each stratum $S \subset \Lambda$, its image $\pi(S) \subset X$ is a stratum,   and the restriction $\pi|_S: S \to \pi(S)$ is a submersion. 
Since $\Lambda $ is isotropic, it follows that $S \subset T^*_{\pi(S)} X$.
 \end{proof}

\subsection{Sheaves}
We fix a field $k$ and work with $k$-linear differential graded (dg) derived categories, or alternatively $k$-linear stable $\infty$-categories.  We write $\dgCat_k$ for the $\oo$-category of  $k$-linear dg categories.
 All functor are derived  without further comment.

Given a  space $X$, we write $\Sh(X)$ for the dg derived category of complexes of sheaves of $k$-modules on $X$.
 We often abuse terminology and refer to objects of $\Sh(X)$ as sheaves.
 
 Given a stratification $\frX$ of $X$, we say   a sheaf $\cF\in \Sh(X)$ is weakly $\frX$-constructible if its restriction $\cF|_S$ to each stratum $S\in \frX$ is cohomologically locally constant. We write $\Sh_\frX \subset \Sh(X)$ for the full dg subcategory  of weakly $\frX$-constructible sheaves. 
 We say   a sheaf $\cF\in \Sh(X)$ is weakly constructible if $\cF$ is weakly $\frX$-constructible for some stratification $\frX$. 
 Throughout the paper, we will  work only with weakly constructible sheaves.

\subsubsection{Useful lemmas}

 \begin{lemma}\label{l: useful 1}
 Let $M$ be a manifold with $X\subset M$ a  compact subset.
 
 Let $f:M\to \RR_{\geq 0}$ be a function.
 
 Suppose $\frX$ is a Whitney stratification of $X$ such that $X \cap \{ f = 0\}$ is a union of strata.
 Let $\frX_{>0}$ and $\frX_{>\epsilon}$ denote the induced stratifications of $X \cap \{ f > 0\}$
 and $X \cap \{ f > \epsilon\}$ respectively.

 Then for all sufficient small $\epsilon >0$, restriction is an equivalence
$$
\xymatrix{
\Sh_{\frX_{>0}}(X \cap \{ f > 0\}) \ar[r]^-\sim & \Sh_{\frX_{>\epsilon}}(X \cap \{f> \epsilon\}) 
}
$$ 
compatible with global sections.
 \end{lemma}
 
 \begin{proof}
By Lemma~\ref{l:trans}, we may choose $\epsilon>0$ so that in the region $f^{-1}(0, 2\epsilon)$, the restriction of $df$ to any stratum of $\frX$ is non-zero. Then by  Thom's first isotopy lemma, there is a stratified isotopy of 
$X \cap \{ f > 0\}$ to $X \cap \{ f > \epsilon\}$.
 \end{proof}

 \begin{lemma}\label{l: useful 2}
 Let $M$ be a manifold with $X\subset M$ a  compact subset.
 
 Let $f:M\to \RR_{\geq 0}$ be a function. 
 
 Consider the natural inclusions
 $i:Y = X\cap \{ f = 0\} \to X$, $j:U = X\cap \{ f > 0\} \to X$,
 $s_\epsilon:S_\epsilon = X\cap \{ f = \epsilon\} \to X$, for $\epsilon>0$.
 
 Suppose $\frX$ is a Whitney stratification of $X$ such that $Y$ is a union of strata.

 Then for $\cF\in \Sh_\frX(X)$, and all sufficient small $\epsilon >0$, there is a canonical equivalence
$$
\xymatrix{
\Gamma(Y, i^*j_*\cF) \simeq \Gamma(S_{\epsilon}, s_\epsilon^*\cF)
}
$$ 

 \end{lemma}
 
 \begin{proof}
By Lemma~\ref{l:trans}, we may choose $\epsilon>0$ so that in the region $f^{-1}(0, 2\epsilon)$, the restriction of $df$ to any stratum of $\frX$ is non-zero. Then by  Thom's first isotopy lemma, there is a stratified deformation retraction of 
$X \cap \{ 2\epsilon > f > 0\}$ into $S_\epsilon = X \cap \{ f = \epsilon\}$. Thus we have
$$
\xymatrix{
\Gamma(Y, i^*j_*\cF) \simeq  \Gamma(X \cap \{ 2\epsilon > f > 0\}, \cF) \simeq  \Gamma(S_{\epsilon}, s_\epsilon^*\cF)
}
$$ 
 \end{proof}

 \begin{remark}
 Both Lemmas~\ref{l: useful 1} and~\ref{l: useful 2} and their proofs hold more generally when $X \subset M$ is not necessarily compact but $f|_X$ is proper over $[0, \delta)$, for some $\delta>0$.
 \end{remark}

 \subsubsection{Singular support}
 To a weakly constructible sheaf $\cF\in \Sh(X)$, following Kashiwara-Schapira~\cite{KS},  we assign its  singular support  $\ssupp(\cF) \subset T^*X$.  Recall $\ssupp(\cF)$ is a closed, conic, Lagrangian subset of $T^*X$, and
if $\cF$ is weakly $\frX$-constructible for a   stratification $\frX$, then $\ssupp(\cF) \subset \bigcup_{X_i \in \frX} T^*_{X_i} X$.

 A codimension zero submanifold with corners $U \subset  X$ is locally modeled on 
 $$
 \RR^{k} \times \RR_{\geq 0}^{n-k} \subset \RR^n
 $$
 for varying~$k$. We regard $U$ as stratified by its open  interior $U_0 \subset U$ and boundary faces $U_i \subset \partial U$ of codimension $i>0$.
 For each such stratum $U_i \subset U$, its outward conormal cone $T^+_{U_i} X \subset   T^*X$ is locally modeled on  
 $$
 (\RR^{k} \times \{0\}) \times (\{0\} \times (\RR_{\leq 0})^{n-k}) \subset \RR^n \times \RR^n \simeq T^*\RR^n
 $$
  In particular, for the interior $U_0 \subset U$, we recover the zero-section $T^+_{U_0} X \simeq  U_0 \subset   T^*U_0 \subset T^*X$.
 
 Given a  codimension zero submanifold with corners $U \subset  X$, we refer to 
 $$
 \Lambda_U = \bigcup_i T^+_{U_i} X \subset T^*X
 $$ as the {\em outward conormal Lagrangian} of $U$.
 
  Note if $i_0: U_0 \to X$ denotes the inclusion of the interior of $U$, then $\ssupp(i_! k_{U_0}) = \Lambda_U$, and
   for any $\cF\in \Sh(X)$, adjunction gives a natural equivalence $\Gamma(U_0, \cF) \simeq \Hom(i_!k_{U_0}, \cF)$. 
  If the intersection $\ssupp(\cF) \cap \Lambda_{U}$ lies in the zero-section, then by the non-characteristic propagation lemma of Kashiwara-Schapira~\cite{KS}, 
  for  perturbations $U_0^t \subset X$ of the open $U_0$, the sections $\Gamma(U_0^t, \cF)$ are locally constant in the parameter $t$.


\section{Nearby cycles}

We first review the traditional construction of nearby cycles over a one-dimensional base, then develop a framework to organize iterated nearby cycles over flags in higher-dimensional bases.
 
\subsection{One-dimensional base}

Consider the diagram
\beq\label{eq:basic diag}
\xymatrix{
\tilde D^\times \ar[r]^-p & D^\times \ar[r]^-j &  D & \ar[l]_-i \{0\}
}
\eeq
with
$ D = \{ z\in \CC \, |\, |z|^2 <1\}$  the open unit disk, $D^\times = \{ z\in \CC \, |\, 0 < |z|^2 <1\}$ the punctured open unit disk,  
 $i, j$ the natural inclusions, and $p$ the universal cover
$$
\xymatrix{
 p:\tilde D^\times = \HH = \{z\in \CC \, |\,  \Im(z) >0\}   \ar[r] &   D^\times
& 
p(z) = e^{iz}
}
$$
with respect to base points $z_0 = e^{-1} \in D$, and $\tilde z_0 = i\in \tilde D^\times$.

Let $f:X \to D$ be a map of complex manifolds, and consider the diagram with Cartesian squares
$$
\xymatrix{
\ar[d] \tilde X^\times \ar[r]^-{p_X} &  \ar[d] X^\times \ar[r]^-{j_X} \ar[d] &  X  \ar[d]^-f& \ar[l]_-{i_X} X_0 \ar[d] \\
\tilde D^\times \ar[r]^-p & D^\times \ar[r]^-j &  D & \ar[l]_-i \{0\}
}
$$

\begin{defn}
The nearby cycles functor for $f :X\to D$ is defined by 
$$
\xymatrix{
\psi_f:\Sh(X^\times) \ar[r] & \Sh(X_0) & \psi_f = i_X^*j_{X*}p_{X*} p_X^* 
}
$$
\end{defn}

\begin{remark}
For future arguments, it will be convenient to separate the distinct aspects of the nearby cycles. Let us 
define the unwinding functor
$$
\xymatrix{
u_f:\Sh(X^\times) \ar[r] & \Sh(X^\times) &  u_f = p_{X*} p_X^* 
}
$$
and  naive nearby cycles
 $$
\xymatrix{
\nu_f:\Sh(X^\times) \ar[r] & \Sh(X_0) & \nu_f = i_X^*j_{X*}
}
$$
so that we have $\psi_f = \nu_f \circ u_f$.

\end{remark}

Let us briefly mention some standard properties of nearby cycles.

\subsubsection{Monodromy}

Note that $ \Aut(p_X) \simeq \ZZ$ generated by the deck transformation $\tau_X:\tilde X^\times  \to \tilde X^\times $, $\tau_X(x, z) = (x, z + 2\pi)$. From the composition identity $\tau_X^* p_X^* \simeq p_X^*$ and base-change identity 
$p_{X*} \simeq p_{X*}\tau_X^*$, we obtain an automorphism
$$
\xymatrix{
m_X:u_f = p_{X*}p_X^* \ar[r]^-\sim & p_{X*}\tau^*p_X^* \ar[r]^-\sim & p_{X*}p_X^* = u_f
}
$$
The monodromy automorphism of the nearby cycles  $\psi_f = \nu_f \circ u_f$ is defined by 
$$
m_f = \nu_f(m_X)\in \Aut(\psi_f)
$$

\subsubsection{Constructibility}


Suppose $\frX = \{X_a\}_{a\in A}$ is a stratification of $X$ so that $X_0$ is a union of strata.
Let $\frX^\times$, $\frX_0$ denote the respective induced stratifications of $X^\times$,
$X_0$.

If $\cF\in \Sh(X^\times)$ is weakly $\frX^\times$-constructible, then $\psi_f\cF\in \Sh(X_0)$ is weakly $\frX_0$-constructible. 

Suppose further that $f$ is compatible with $\frX$ and
 the stratification $\frD = \{\{0\}, D^\times\}$ of $D$.

If $\cF\in \Sh(X^\times)$ is   $\frX^\times$-constructible, then
 $\psi_f \cF \in \Sh(X_0)$ is $\frX_0$-constructible.
 
 \subsubsection{Compatibilities}

Consider a map of complex manifolds $g:X'\to X$.

If $g$ is proper, then there is a natural equivalence $g_! \psi_{g\circ f}  \simeq \psi_f g_!$.
 
If $g$ is smooth, then there is a natural equivalence $g^* \psi_f  \simeq \psi_{g\circ f} g^*$.

 The shifted nearby cycles $\psi_f[-1]$ commutes with Verdier duality $\cD\psi_f[-1] \simeq \psi_f[-1]\cD$.


\subsection{Higher-dimensional base}\label{ss:hd base}

Fix $n>0$, and write $[n]  = \{1, \ldots, n\}$. 


For $a\subset [n]$ nonempty,  taking maps from $a$ into the diagram~\eqref{eq:basic diag} provides a diagram
\beq\label{eq:prod diag}
\xymatrix{
\tilde D^\times(a) \ar[r]^-{p(a)} & D^\times(a) \ar[r]^-{j(a)} &  D(a) & \ar[l]_-{i(a)} \{0\}
}
\eeq
Note that $D([n]) = D^n$. 

For $a,  b  \subset [n]$ disjoint with $a$ nonempty, taking the product of ~\eqref{eq:prod diag} with $D^\times(b)$ provides a diagram
\beq\label{eq:next prod diag}
\xymatrix{
\tilde D^\times(a) \times D^\times(b) \ar[r]^-{p(a, b)} & D^\times(a)  \times D^\times(b)\ar[r]^-{j(a, b)} &  D(a)  \times D^\times(b) & \ar[l]_-{i(a, b)}  \{0\} \times D^\times(b)
}
\eeq
Note that $  D^\times(a)  \times D^\times(b) = D^\times(a\cup b)$.
 Note if $b = \emptyset$, then $D^\times(\emptyset)$ is a point and \eqref{eq:next prod diag} coincides with \eqref{eq:prod diag}.

Now let $f:X \to D^n = D([n])$ be a map of complex manifolds. For $a,  b  \subset [n]$ disjoint, taking fiber products with \eqref{eq:next prod diag} over $D([n])$ provides a diagram with Cartesian squares
$$
\xymatrix{
\ar[d] \tilde X^\times(a, b) \ar[r]^-{p_X(a, b)} &  \ar[d] X^\times(a\cup b)  \ar[r]^-{j_X(a, b)} \ar[d] &  X(a, b)  \ar[d]^-f& \ar[l]_-{i_X(a, b)} X^\times(b) \ar[d] \\
\tilde D^\times(a) \times D^\times(b) \ar[r]^-{p(a, b)} & D^\times(a)  \times D^\times(b)\ar[r]^-{j(a, b)} &  D(a)  \times D^\times(b) & \ar[l]_-{i(a, b)}  \{0\} \times D^\times(b)
}
$$
Note if $b = \emptyset$, then $D^\times(\emptyset)$ is a point,
and $X^\times(\emptyset) = X_0 = f^{-1}(0)$.

\begin{defn}
 For $f :X\to D^n$, and  $a, b \subset [n]$ disjoint with $a$ nonempty, 
we define the $a \cup b  \rightsquigarrow b$ nearby cycles by
$$
\xymatrix{
\psi^{a \cup b}_b:\Sh(X^\times(a \cup b )) \ar[r] & \Sh(X^\times (b) ) & \psi_b^{a \cup b} = i_X(a, b)^*j_{X}(a, b)_*p_{X}(a, b)_* p_X(a, b)^*
}
$$
\end{defn}

\begin{remark}\label{rem:separate}
It will be convenient to separate the distinct aspects of the $a \cup b  \rightsquigarrow b$ nearby cycles. To that end, we 
define the  $a \cup b  \rightsquigarrow b$ unwinding functor
$$
\xymatrix{
u_b^{a \cup b}:\Sh(X^\times(a \cup b )) \ar[r] & \Sh(X^\times(a \cup b )) &  u_b^{a \cup b} = p_{X}(a, b)_* p_X(a, b)^* 
}
$$
and   $a \cup b  \rightsquigarrow b$ naive nearby cycles
 $$
\xymatrix{
\nu^{a \cup b}_b:\Sh(X^\times(a \cup b )) \ar[r] & \Sh(X^\times (b) ) & \nu^{a \cup b}_b = i_X(a, b)^*j_{X}(a, b)_*
}
$$
so that we have $\psi^{a \cup b}_b = \nu^{a \cup b}_b \circ u_b^{a \cup b}$.

\end{remark}

Note that $ \Aut(p_X(a, b)) \simeq \ZZ^a$ generated by the deck transformation $\tau_X:\tilde X^\times  \to \tilde X^\times $, $\tau_X(x, z) = (x, z + 2\pi)$ applied to each component. Following the construction for traditional nearby cycles recalled above, we obtain a homomorphism 
$$
\xymatrix{
m_X(a, b):\ZZ^a \ar[r] &  \Aut(u^{a\cup b}_b)
}
$$
The monodromy of the nearby cycles  $\psi^{a \cup b}_b = \nu^{a \cup b}_b \circ u_b^{a \cup b}$ is defined
to be 
$$
m^{a \cup b}_b =  \nu^{a \cup b}_b(m_X(a, b)):\ZZ^a \to \Aut(\psi^{a\cup b}_b)
$$

\subsubsection{Iterated nearby cycles}

By a  partial flag $a_\bullet$ in $[n]$, we will mean a sequence of distinct subsets 
$$
\emptyset  = a_0 \subset a_1 \subset \cdots \subset a_{k+1} = [n]
$$
We call $k = \ell(a_\bullet)$ the length of the flag.

\begin{defn}\label{def:iterated nc}
 For $f :X\to D^n$, and a  partial flag $a_\bullet$  in $[n]$, we define the $a_\bullet$-iterated nearby cycles by
$$
\xymatrix{
\psi(a_\bullet) = \psi^{a_1}_\emptyset \circ  \psi^{a_2}_{a_{1}}  \circ  \cdots \circ \psi^{a_{k}}_{a_{k-1}}
\circ \psi^{[n]}_{a_k}:
\Sh(X^\times) \ar[r] & \Sh(X_0 ) 
}
$$ 
where we write $X^\times = X^\times([n]) = f^{-1}((D^\times)^n)$, $X_0 = X^\times(\emptyset) =  f^{-1}(0)$.
%
%
\end{defn}

We define the monodromy of the  $a_\bullet$-iterated nearby cycles 
$$
m(a_\bullet) :\ZZ^n \to \Aut(\psi(a_\bullet))
$$
to be the product of the monodromies of its constituent nearby cycles.

\subsubsection{Isolating the unwinding}

Let us next isolate once and for all the unwinding part of the iterated nearby cycles. Consider the cover
$$ 
\xymatrix{
p_n:\tilde X^\times = X \times_{D^n}  (\tilde D^\times)^n \ar[r] &  X \times_{D^n} (D^\times)^n =  X^\times 
}
$$
and define the maximal unwinding functor
$$
\xymatrix{
u_n:\Sh(X^\times) \ar[r] & \Sh(X^\times) & u_n = p_{n*} p_n^*
}
$$
with its natural monodromy 
$$
m_n:\ZZ^n \to \Aut(u_n)
$$
In terms of prior notation, we have   $u_n = u^{[n]}_\emptyset$ and $m_n = m_X([n], \emptyset)$.

 For a  partial flag $a_\bullet$  in $[n]$,  define the $a_\bullet$-iterated naive nearby cycles by
$$
\xymatrix{
\nu(a_\bullet) = \nu^{a_1}_\emptyset \circ  \nu^{a_2}_{a_{1}}  \circ  \cdots \circ \nu^{a_{k}}_{a_{k-1}}
\circ \nu^{[n]}_{a_k}:
\Sh(X^\times) \ar[r] & \Sh(X_0 ) 
}
$$

   
\begin{lemma}\label{l: unwind}
 The $a_\bullet$-iterated  nearby cycles is canonically equivalent to the maximal unwinding functor followed by the
 $a_\bullet$-iterated naive nearby cycles:
$$
\xymatrix{
\psi(a_\bullet) \simeq 
\nu(a_\bullet) \circ u_n
}
$$
Moreover, the equivalence respects the natural $\ZZ^n$-monodromy actions.

\end{lemma}

\begin{proof}
For $a, b \subset [n]$ disjoint with $a$ nonempty, we will construct 
an equivalence
\beq\label{eq:inductive step}
\xymatrix{
u_\emptyset^b \nu^{a \cup b}_b u^{a\cup b}_b   \simeq
 \nu^{a\cup b}_b u^{a \cup b}_\emptyset
}
\eeq
compatible with the natural $\ZZ^{a\cup b}$-monodromy actions. 
Applying \eqref{eq:inductive step} successively to the terms of $\psi(a_\bullet)$  results in the asserted equivalence. 

To construct \eqref{eq:inductive step}, consider the diagram with Cartesian squares and vertical covering maps
\beq\label{eq: diag chase}
\xymatrix{
\ar[d]  \tilde D^\times(a) \times \tilde D^\times(b) \ar[r]^-{} &  \ar[d] D^\times(a) \times \tilde D^\times(b)  \ar[r]^-{} \ar[d] &  D(a) \times \tilde D^\times(b)  \ar[d]& \ar[l]_-{} \{0\} \times \tilde D^\times(b) \ar[d] \\
\tilde D^\times(a) \times D^\times(b) \ar[r]^-{p(a, b)} & D^\times(a)  \times D^\times(b)\ar[r]^-{j(a, b)} &  D(a)  \times D^\times(b) & \ar[l]_-{i(a, b)}  \{0\} \times D^\times(b)
}
\eeq
Take the fiber product of \eqref{eq: diag chase} with $f|_{D(a) \times D^\times(b)}$ to obtain a diagram with analogous properties. The sought-after equivalence now results from a  chase in this last diagram. One uses: 1)  standard identities for compositions,  2) smooth base-change for pullback along the vertical  maps  in the left and middle square, and 
3) a special instance of base-change, appearing in Lemma~\ref{l:covering bc} immediately following,  for the pushforward  along the vertical maps  in the right square.

Finally, one can observe the diagrams have compatible deck transformations, hence the constructed equivalence is  compatible with the natural $\ZZ^{a\cup b}$-monodromy actions. 
\end{proof}

\begin{lemma}\label{l:covering bc}
Let $p:\tilde Y \to Y $ be a covering of manifolds, and $v:V\to Y$ the inclusion of a locally closed submanifold. Consider the Cartesian square
$$
\xymatrix{
\ar[d]_-{p_V}\tilde V \ar[r]^-{\tilde v} & \tilde Y\ar[d]_p\\
 V \ar[r]^-{v} & Y 
}
$$
For a weakly constructible complex $\cF\in \Sh(Y)$, with pullback $p^*\cF\in \Sh(\tilde Y)$, the natural base-change map is an equivalence
\beq\label{eq:covering bc map}
\xymatrix{
v^*p_*p^*\cF \ar[r]^-\sim & p_{V*} \tilde v^* p^*\cF
}
\eeq
\end{lemma}

\begin{proof}
The assertion is local so it suffices to assume $p:\tilde Y = Y \times S\to Y$ is the projection where $S$ is a discrete space. Then \eqref{eq:covering bc map} takes the form
$$
\xymatrix{
v^*(\prod_S \cF) \ar[r] & \prod_S v^*\cF
}
$$
Since $\cF\in \Sh(Y)$ is weakly constructible, $\prod_S \cF \in \Sh(Y)$ is as well,  hence the colimit diagrams in the calculation of $v^*(\prod_S \cF)$ stabilize. Thus sections of $v^*(\prod_S \cF)$ on opens are simply calculated by  sections of $\prod_S \cF$ on suitable opens, and thus  commute with the product.
\end{proof}

\subsubsection{Lax compatibility of nearby cycles}
 
Now we will relate the various iterated nearby cycles. One could also use the technology of vanishing topoi. 

Let $\Flags([n])$ denote the set of partial flags $a_\bullet$  in $[n]$. 
Equip  $\Flags([n])$  with the partial order $a^1_\bullet \leq  a^2_\bullet$ when $a^1_\bullet$ results from  $a^2_\bullet$
by forgetting some of its subsets. We will regard $\Flags([n])$ as a category with a single morphism $a^1_\bullet \to  a^2_\bullet$
whenever $a^1_\bullet \leq  a^2_\bullet$ and the evident composition.

\begin{lemma}\label{l: naive lax}
The assignment of naive  iterated nearby cycles 
$$
\xymatrix{
\Flags([n]) \ar[r] &  \Fun_{\dgCat_k}(\Sh(X^\times), \Sh(X_0 )) & a_\bullet \ar@{|->}[r] & \nu(a_\bullet)
}
$$
naturally  extends to a map of $\oo$-categories:  
for each $a^1_\bullet \leq  a^2_\bullet$, there is a canonical map of functors
$$
\xymatrix{
r^{a^1_\bullet}_{a^2_\bullet}:\nu(a^1_\bullet)\ar[r] & \nu(a^2_\bullet)  
}
$$
along with coherent compositional identities 
$$
\xymatrix{
r^{a^1_\bullet}_{a^\ell_\bullet} \simeq r^{a^{\ell-1}_\bullet}_{a^\ell_\bullet} \circ \cdots \circ 
r^{a^1_\bullet}_{a^2_\bullet}
}
$$
for chains $a^1_\bullet \leq \cdots \leq a^\ell_\bullet$.
\end{lemma}

\begin{proof} 
We will construct the map $r^{a^1_\bullet}_{a^2_\bullet}$ when $a^1_\bullet \leq a^2_\bullet$ with $\ell(a^2_\bullet) = \ell(a^1_\bullet) + 1$. Coherent compositional identities result from diagram chases  indicated below.

So suppose $a^1_\bullet$, $a^2_\bullet$ are  given respectively by
$$
\xymatrix{
\emptyset  = a_0 \subset a_1 \subset \cdots \subset \wh a_i \subset \cdots \subset a_{k+1} = [n]
&
\emptyset  = a_0 \subset a_1 \subset \cdots \subset a_{k+1} = [n]
}
$$
where $\wh a_i$ means we omit $a_i$ from $a^1_\bullet$.
Thus both functors $\nu(a^1_\bullet)$, $\nu(a^2_\bullet)$ begin with the composition
$$
\xymatrix{
  \nu^{a_{i+2}}_{a_{i+1}}  \circ  \cdots \circ \nu^{a_{k}}_{a_{k-1}}
\circ \nu^{[n]}_{a_k}:
\Sh(X^\times) \ar[r] & \Sh(X^\times(a_{i+1})) 
}
$$ 
and similarly end with the composition
$$
\xymatrix{
  \nu^{a_1}_\emptyset \circ  \nu^{a_2}_{a_{1}}  \circ  \cdots 
 \nu^{a_{i-1}}_{a_{i-2}}:
\Sh(X^\times(a_{i-1})) \ar[r] & \Sh(X_0) 
}
$$

Set $c= a_{i-1}$, $b = a_i \setminus a_{i-1}$, $a = a_{i+1} \setminus a_i$. Then it suffices to give a map of functors
$$
\xymatrix{
\nu^{a \cup b \cup c}_{c}  \ar[r] &
\nu^{b\cup c}_c  \circ \nu^{a \cup b \cup c}_{b \cup c}  :
\Sh(X^\times(a \cup b \cup c)) \ar[r] & \Sh(X^\times(c)) 
}
$$
Returning to the definitions, we seek  a map of functors
$$
\xymatrix{
  i_X(a\cup b, c)^*j_{X}(a \cup b, c)_*       \ar[r] & 
i_X(b, c)^*j_{X}(b, c)_*   i_X(a, b\cup c)^*j_{X}(a, b\cup c)_*  
}
$$

Consider the diagram with Cartesian square
\beq
\label{eq:big diag}
\xymatrix{
\ar[dr]_-{j(a\cup b, c)}D^\times(a) \times D^\times(b) \times D^\times(c) \ar[r]^-{j(a,b\cup c)} & 
\ar[d]^-j D(a) \times D^\times(b) \times D^\times(c) & \ar[l]_-{i(a,b\cup c)} 
\ar[d]^-{j(b, c)} \{0\} \times D^\times(b) \times D^\times(c) \\
& D(a) \times D^\times(b) \times D^\times(c) 
& \ar[l]_-i \{0\} \times D(b) \times D^\times(c) \\
& 
& 
\ar[ul]^-{i(a \cup b, c)} \{0\} \times \{0\} \times D^\times(c) \ar[u]_-{i(b, c)}
}
\eeq

Taking the fiber product of \eqref{eq:big diag}  with  $f:X\to D^n = D([n])$ gives an analogous diagram 
\beq
\label{eq:big X diag}
\xymatrix{
\ar[drr]_-{j_X(a\cup b, c)}X^\times(a) \times X^\times(b) \times X^\times(c) \ar[rr]^-{j_X(a,b\cup c)} && 
\ar[d]^-{j_X} X(a) \times X^\times(b) \times X^\times(c) && \ar[ll]_-{i_X(a,b\cup c)} 
\ar[d]^-{j_X(b, c)} X_0 \times X^\times(b) \times X^\times(c) \\
&& X(a) \times X^\times(b) \times X^\times(c) 
&& \ar[ll]_-{i_X} X_0 \times X(b) \times X^\times(c) \\
&& 
&& 
\ar[ull]^-{i_X(a \cup b, c)} \{0\} \times \{0\} \times D^\times(c) \ar[u]_-{i_X(b, c)}
}
\eeq

 The commutative triangles give composition identities
$$
\xymatrix{
  j_{X}(a \cup b, c)_*   \simeq j_{X *} j_X(a, b \cup c)_*  &
   i_X(a\cup b, c)^* \simeq i(b, c)^* i_X^*
}
$$
Finally, base-change in the upper right square gives the sought-after map of functors
 $$
\xymatrix{
  i_X(a\cup b, c)^*j_{X}(a \cup b, c)_*  \simeq i(b, c)^* i_X^*  j_{X *} j_X(a, b \cup c)_*  
  \ar[r] &   
  i_X(b, c)^*j_{X}(b, c)_*   i_X(a, b\cup c)^*j_{X}(a, b\cup c)_*  
}
$$

To obtain coherent compositional identities, one can assemble a diagram composed of triangles as in~\eqref{eq:big X diag}
and similarly assemble natural maps.
\end{proof}

Finally, Lemmas~\ref{l: unwind} and~\ref{l: naive lax} 
immediately imply the following.

Let $T^n = \RR^n/\ZZ^n$ be the $n$-torus. 
Let $\Sh(X_0 )^{\otimes T^n}$ denote the category of objects $\cF\in  \Sh(X_0 )$ together with a map
$\ZZ^n \to \Aut(\cF)$.

\begin{prop}\label{prop: lax}

The assignment of  iterated nearby cycles 
$$
\xymatrix{
\Flags([n])  \ar[r] & \Fun_{\dgCat_k}(\Sh(X^\times), \Sh(X_0 )^{\otimes T^n})
&
a_\bullet  \ar@{|->}[r] & \psi(a_\bullet)
}
$$ naturally extends to a map of $\oo$-categories. 
\end{prop}

Our main result, presented in the next section, gives a  criterion for when the maps in the functor of Proposition~\ref{prop: lax}, as constructed by base-change in Lemma~\ref{l: naive lax}, are equivalences.


\section{Microlocal criterion}


The aim of this section is to state and prove our main result Theorem~\ref{thm: main}. 

  \subsection{Hypotheses}\label{s:hypo}
  
  We detail here the hypotheses that go into  Theorem~\ref{thm: main}.

\subsubsection{Non-characteristic Lagrangians}
  
   To a map of manifolds  $f:X\to Y$, we have the  Lagrangian correspondence between cotangent bundles
$$
\xymatrix{
T^*Y & \ar[l]_-p f^*(T^*Y)  \ar[r]^-{i} & T^*X
}
$$
where $f^*(T^*Y) = T^*Y \times_{Y} X$ is the pulled back bundle, 
 $p$ is the evident projection, and $i$ is the pullback of covectors. 
 
%

\begin{defn} Let $\Lambda \subset T^*X$ be a subset.

 For a map $f:X\to Y$, we say that $\Lambda$ is {\em $f$-non-characteristic} if the intersection 
$$ \Lambda \cap i(p^{-1}(T^*Y))
$$ lies in the zero-section of $T^*X$.
\end{defn}

\begin{remark}
For $Z \subset X$ a submanifold, the conormal bundle $T^*_Z X \subset T^*X$ is $f$-non-characteristic  if and only if 
the restriction $f|_Z:Z \to Y$ is a submersion.
\end{remark}

\subsubsection{Thom condition}\label{sss:thom}

Now suppose the map of manifolds $f:X\to Y $ is a submersion. Then the pullback of covectors  $i:f^*(T^*Y)\to T^*X $ is injective, and we have a short exact exact sequence of vector bundles
$$
\xymatrix{
0 \ar[r] & f^*(T^*Y)  \ar[r]^-{i} & T^*X \ar[r]^-\Pi  & T^*_f \ar[r] & 0
}
$$
where $T^*_f$ is the relative cotangent bundle. Note for any $y\in Y$, with fiber $X_y = f^{-1}(y)$, we have a canonical identification $T^*_f|_{X_y}  \simeq T^*X_y$.

Given  a submersion $f:X\to Y$, and a  subset $\Lambda \subset T^*X$, we refer to its image $\Lambda_f = \Pi(\Lambda) \subset T^*_f$ in the relative cotangent bundle as  the {\em $f$-projection} of $\Lambda$. 
 We will be interested in the closure
of the $f$-projection
$$
 \ol\Lambda_f \subset T^*_f
 $$  
 in particular
 its restriction to   fibers $X_y$, for $y\in Y$, which we denote by
 $$
\xymatrix{
 \ol \Lambda_{f, y} = \ol\Lambda_f \times_Y {X_y} \subset T^*_f|_{X_y} \simeq T^* X_y
 }
 $$

%

%
%
%
%

\begin{defn}  Let $\Lambda \subset T^*X$ be a conic Lagrangian, and
 $f:X\to Y$ a submersion.

We say that $\Lambda$ is {\em $f$-Thom} at a point $y\in Y$
 if  the restriction 
of the closure of the $f$-projection  $\ol\Lambda_{f, y}  \subset T^*X_y$ is isotropic.

We say that $\Lambda$ is  $f$-Thom  if it is $f$-Thom at all $y\in Y$.

\end{defn}

\begin{remark}\label{rem:unclosed}
For $\Lambda \subset T^*X$ a conic Lagrangian, and $y\in Y$, the (not necessarily closed) restriction of the $f$-projection
$$
\xymatrix{
  \Lambda_{f, y} = \Lambda_f \times_S {X_y} \subset T^*_f|_{X_y} \simeq T^* X_y 
 }
 $$
is always isotropic. To see this, consider the inclusion of a fiber $X_y \subset X$, for $y\in Y$, and the associated Lagrangian correspondence
$$
\xymatrix{
T^*X & \ar[l]_-p T^*X|_{X_y}  \ar[r]^-{i} & T^*X_y
}
$$
Then alternatively we have $\Lambda_{f, y} = i(p^{-1}(\Lambda))$, hence it is isotropic by general considerations about Lagrangian correspondences. 
\end{remark}

\begin{ex}

We mention here a useful situation where one can verify the $f$-Thom condition. We thank Misha Finkelberg for bringing this application to our attention. Suppose $g:G\to Y$  is a smooth group-scheme acting on $f:X\to Y$. Let 
$\frg^*\to Y$ be the dual of Lie algebra, and $\mu:T^*_f \to \frg^*$ the moment map for the induced action.

Suppose $G_y$ acts on $X_y$ with discretely many orbits. Then $\mu^{-1}(0)|_y\subset T^*_f|_y \simeq T^*X_y$ is
the union of the conormals to the orbits and in particular Lagrangian. 

Let $\Lambda \subset T^*X$ be a conic Lagrangian  such that $\Lambda_f \subset T^*_f$ lies in $\mu^{-1}(0) \subset T^*_f$. For example,  $\Lambda$ could be the singular support of a sheaf  $\cF\in \Sh(X)$ that is locally constant along the $G$-orbits. Then we immediately have that $\Lambda$ is  $f$-Thom at $y$ since $\ol \Lambda_{f, y} \subset \mu^{-1}(0)|_y$ which is Lagrangian. 

\end{ex}

Before continuing on, 
let us mention a traditional source of Lagrangians satisfying the Thom condition. 

Recall (see ~\cite{Math}) Thom's condition $A_f$ (relative Whitney condition $A$) for a map of manifolds $f:X\to Y$.
Given
 submanifolds $X_0, X_1 \subset X$ with $f|_{X_0}, f|_{X_1}$ of constant rank, one says that the pair $(X_0, X_1)$ satisfies  condition $A_f$ at a point $x_0\in X_0$ if the following holds.
 Let $x_i \in X_1$  be a sequence  of points  converging to $x_0$ such that the sequence of planes $\ker((df|_{X_1})_{x_i}) \subset T_{x_i} X_1$ converges to a plane $\tau \subset T_{x_0} X$ in the Grassmannian of the tangent bundle. Then $\ker((df|_{X_0})_{x_0})  \subset \tau$.
 One says the pair $(X_0, X_1)$ satisfies condition $A_f$ if it satisfies condition $A_f$ at all points  $x_0\in X_0$.
 
Now suppose $f:X\to Y$ is a submersion.  Then by  passing from relative tangent subspaces to relative conormal quotient spaces, 
we may reinterpret condition $A_f$  in the following form.  

\begin{lemma}\label{l:reform thom}
The pair $(X_0, X_1)$ satisfies condition $A_f$ if and only if 
$$
\xymatrix{
\ol{\Pi(T^*_{X_1} X)}|_{X_0} \subset \Pi(T^*_{X_0} X)
}
$$
\end{lemma}

\begin{proof}
Recall we write $T_f = \ker(df)$ for the relative tangent bundle, dual to the relative cotangent bundle $T^*_f$. 
Then inside the dual relative bundles, we have $\ker(df|_{X_0})^\perp = \Pi(T^*_{X_0} X)$, $\ker(df|_{X_1})^\perp = \Pi(T^*_{X_1} X)$. 

Suppose given the ingredients in the definition of conditions $A_f$. 
If  $\ol{\Pi(T^*_{X_1} X)}|_{X_0} \subset \Pi(T^*_{X_0} X)$, then 
$  \lim_{x_i} \Pi(T^*_{X_1} X)_{x_i}  \subset \Pi(T^*_{X_0} X)_{x_0}$, and so 
$ \lim_{x_i} \ker(df|_{X_1})_{x_i} \supset \ker(df|_{X_0})_{x_0}$.

Conversely, for any point $(x_0, \xi_0) \in \ol{\Pi(T^*_{X_1} X)}|_{X_0}$, by definition $x_0 = \lim x_i$ for some $x_i \in X_1$. By passing to a subsequence, we may assume $\ker((df|_{X_1})_{x_i}) \subset T_{x_i} X_1$ converges to a plane $\tau \subset T_{x_0} X$. If $\ker((df|_{X_0})_{x_0})  \subset \tau$, then there are $(x_i, \xi_i) \in \Pi(T^*_{X_1} X)_{x_i}$ arbitrarily close to $(x_0, \xi_0)$.
\end{proof}

\begin{prop}\label{p:thom=thom}
Let $f:X\to Y$ be a submersion  with compatible stratifications $\frX, \frY$. 

If each pair of strata of $\frX$ satisfies condition $A_f$, then the conormal Lagrangian $\Lambda_{\frX} = \bigcup_i T^*_{X_i} X$ is $f$-Thom.
\end{prop}

\begin{proof}
By Lemma~\ref{l:reform thom} it suffices to see $\Pi(T^*_{X_i} X)|_{X_y}$, for $y\in Y$, is isotropic. This  always holds as explained in Remark~\ref{rem:unclosed}. 
\end{proof}

\begin{remark}\label{rem:gv}
Conversely, we have the following which we learned from Misha Grinberg and Kari Vilonen.
Set $U = X \setminus X_y$.
Given a conic Lagrangian  $\Lambda \subset T^*U$, we can always find  submanifolds $U_i\subset U$ such that $\Lambda \subset \bigcup_i \ol{T^*_{U_i} U}$. Then one observes $\Lambda$ is $f$-Thom at $y$ if and only if there exists a stratification $\frV = \{V_j\}$ of  $X_y$ such that each pair $(U_i, V_j)$ satisfies condition $A_f$. If the pairs also satisfy Whitney's condition $B$ then one can follow~\cite{Le} and invoke Thom's second isotopy lemma to obtain our main result. Grinberg and Vilonen have pointed out that  in fact Thom's second isotopy lemma may only require Whitney's condition $B$ for pairs mapping to the same stratum of the base. If so, then one can obtain our main result with this method, though our arguments are far more elementary. 
\end{remark}

Finally, let us point out  that the Thom condition is essentially about higher-dimensional bases.

\begin{prop}\label{p:dim 1}
If $f:X\to Y$ is a submersion with $\dim Y = 1$, then any conic Lagrangian $\Lambda\subset T^*X$ is $f$-Thom.

\end{prop}

\begin{proof}
Choose stratifications $\frX, \frY$ compatible with $f$ so that $\Lambda \subset  \Lambda_\frX = \bigcup_i T^*_{X_i} X$. By Prop.~\ref{p:thom=thom}, it suffices to check each pair of strata  of $\frX$ satisfies condition $A_f$.  
This follows immediately from~\cite[Th\'eor\`eme 4.2.1]{BMM} and the  fact that Whitney stratifications are topologically trivial by Thom's first isotopy lemma.
\end{proof}

  
  \subsection{Main result}

Now we are ready to state and prove our main result.   We return to the constructions and notation
of Section~\ref{ss:hd base}. 

Recall 
given a map of complex manifolds $f :X\to D^n$, 
 we write $X^\times = f^{-1}((D^\times)^n)$, $X_0 = f^{-1}(0)$.
Recall by Proposition~\ref{prop: lax},
iterated nearby cycles, with their monodromy, gives a map of $\oo$-categories
$$
\xymatrix{
\psi:\Flags([n])  \ar[r] & \Fun_{\dgCat_k}(\Sh(X^\times), \Sh(X_0 )^{\otimes T^n})
}
$$ 
where $\Flags([n])$ denotes the category of partial flags in $[n]$, 
and $\Sh(X_0 )^{\otimes T^n}$ the category of complexes $\cF\in \Sh(X_0)$ equipped with a map $\ZZ^n\to \Aut(\cF)$.

\begin{theorem}\label{thm: main}

Let $\Lambda \subset T^*X^\times$ be a closed conic Lagrangian, and $f :X\to D^n$ a submersion.

Suppose  $\Lambda$ is (i) $f$-non-characteristic and (ii) $f$-Thom at the origin $0\in D^n$

Let $\Sh_\Lambda(X^\times) \subset \Sh(X^\times)$ be the full subcategory of complexes with singular support contained in $\Lambda$.

Then the restricted map of $\oo$-categories
$$
\xymatrix{
\psi:\Flags([n])  \ar[r] & \Fun_{\dgCat_k}(\Sh_\Lambda(X^\times), \Sh(X_0 )^{\otimes T^n})
}
$$ 
lands in the full sub-groupoid of functors and invertible natural transformations. Concretely,
for any $\cF\in \Sh_\Lambda(X^\times)$, and
each pair of partial flags $a^1_\bullet \leq  a^2_\bullet$ in $[n]$, the canonical map of
iterated nearby cycles functors
$$
\xymatrix{
\psi(a^1_\bullet)(\cF)\ar[r] & \psi(a^2_\bullet)(\cF)
}
$$
is an equivalence.
\end{theorem}

\begin{proof} 
Observe  that the universal unwinding functor 
$$
\xymatrix{
u_n = p_{n*} p_n^*:\Sh(X^\times) \ar[r] &  \Sh(X^\times)
}
$$ evidently preserves singular support since $p_n:\tilde X^\times \to X^\times$ is a cover. Following the construction of  Proposition~\ref{prop: lax}, it thus suffices to prove the theorem for   the map of $\oo$-categories given by iterated naive nearby cycles
$$
\xymatrix{
\nu:\Flags([n]) \ar[r] &  \Fun(\Sh_\Lambda(X^\times), \Sh(X_0 )) 
}
$$
 In other words,
for
each pair of flags $a^1_\bullet \leq  a^2_\bullet$ in $[n]$, it suffices to show the canonical map of
naive nearby cycles functors
$$
\xymatrix{
r^{a^1_\bullet}_{a^2_\bullet}:\nu(a^1_\bullet)\ar[r] & \nu(a^2_\bullet)  
}
$$
is an equivalence when evaluated on objects of $\Sh_\Lambda(X^\times)$.


%
%
%

To check when   $r^{a^1_\bullet}_{a^2_\bullet}$ is an equivalence, it suffices to check on stalks. So we may work locally and assume our given map is the projection $f:X = \CC^m \times D^n\to D^n$. We will unwind the constructions and concretely calculate $r^{a^1_\bullet}_{a^2_\bullet}$ on the stalk at the origin $0\in \CC^m$.
For $\frz> 0 $, we write $B(\frz)  = \{x\in \CC^m \,|\,  |x|^2\leq \frz^2 \}\subset \CC^m$ for the closed ball of radius $\frz$
around the origin, and $S(\frz) = \partial B(\frz) =  \{x\in \CC^m \,|\,  |x|^2= \frz^2 \}\subset \CC^m$ for the sphere of radius $\frz$.

%
%

For each $a\subsetneq [n]$, 
consider the function
  $$
\xymatrix{
r_a:D^n \ar[r] & \RR_{\geq 0} & r_a(z) = \sum_{i \in [n] \setminus a} z_i^2
}
$$

Consider a collection $\frr = (\frr_a)_{a\subsetneq [n]}$ of positive constants.
When choosing them, we will always select $\frr_{a_1}$ before $\frr_{a_2}$ when  $a_1 \subsetneq a_2$, in particular we will always arrange $\frr_{a_2} \ll \frr_{a_1}$ when $a_1 \subsetneq a_2$.
We will say $\frr = (\frr_a)_{a\subsetneq [n]}$ is  sufficiently small for an event  to hold when for any such sequence of sufficiently small choices of  $\frr_a$, for $a\subsetneq [n]$,  the event holds.

Given  a collection $\frr = (\frr_a)_{a\subsetneq [n]}$ of positive constants, 
define the open tube and closed boundary 
$$
\xymatrix{
T_a(\frr_a) = \{ z\in D^n \, |\,  r_a(z) < \frr_a^2 \}
&
S_a(\frr_a) = \partial T_a(\frr_a) =  \{ z\in D^n \, |\, r_a(z) = \frr_a^2\}
}
$$
 Note we may choose $\frr = (\frr_a)_{a\subsetneq [n]}$ sufficiently small so that the submanifolds  $S_a(\frr_a) \subset D^n$,  for $a\subsetneq [n]$, form a transverse collection. We will assume we have done so in what follows.

Consider the compact submanifold with corners defined by
$$
\xymatrix{
P(\frr) = S_\emptyset(r_\emptyset) \setminus (S_\emptyset(r_\emptyset) \bigcap  (\bigcup_{\emptyset \subsetneq a \subsetneq [n]} T_a(\frr_a)))
}
$$
Note that $P(\frr)$  lies in $(D^\times)^n \subset D^n$, and
 its faces  are naturally indexed by partial flags $a_\bullet$ in $[n]$. Namely, to a partial flag  $a_\bullet$ given by
$
\emptyset  = a_0 \subset a_1 \subset \cdots \subset a_{k+1} = [n],
$
we have the closed face
$$
\xymatrix{
P(a_\bullet, \frr) = P(\frr) \cap (\bigcap_{i = 1}^k S_{a_i}(\frr_i)) 
}
$$
 In particular, $P(\frr)$ itself is indexed by the partial flag $\emptyset = a_0 \subset a_1 = [n]$, and its $n!$ many dimension zero corners  are  indexed by complete flags
$
\emptyset  = a_0 \subset a_1 \subset \cdots \subset a_{n} = [n].
$ 
Note  in turn each closed face
$
P(a_\bullet, \frr)  
$
is itself a submanifold with corners, with† closed faces $P(a'_\bullet, \frr)$, for $a'_\bullet \geq a_\bullet$.

The following  holds without assuming the non-characteristic or Thom hypotheses of the theorem. 

\begin{lemma} Fix a weakly constructible complex $\cF\in \Sh(\CC^m \times D^n)$.
 
For sufficiently small $\frz>0$, and then sufficiently small $\frr = (\frr_a)_{a\subsetneq [n]}$, there is a natural identification of the stalk of the $a_\bullet$-iterated naive nearby cycles at the origin
$$
\xymatrix{
\nu(a_\bullet)(\cF)|_0 \ar[r]^-\sim  & \Gamma(B(\frz) \times { P(a_\bullet, \frr)}, \cF) 
}$$
for any partial flag $a_\bullet$ in $[n]$.

Moreover, for any pair of partial flags  $a^1_\bullet \leq a^2_\bullet$, the natural maps form a commutative square
$$
\xymatrix{
\ar[d]_-{r^{a^1_\bullet}_{a^2_\bullet}|_0} \nu(a^1_\bullet)(\cF)|_0 \ar[r]^-\sim  & \Gamma(B(\frz) \times { P(a^1_\bullet, \frr)}, \cF) \ar[d]^-\rho
\\
\nu(a^2_\bullet)(\cF)|_0 \ar[r]^-\sim  & \Gamma(B(\frz) \times { P(a^2_\bullet, \frr)}, \cF) 
}
$$
where $\rho$ is restriction along the inclusion $B(\frz) \times {P(a^2_\bullet, \frr)} \subset B(\frz) \times {P(a^1_\bullet, \frr)}$.
\end{lemma}

\begin{proof}

Without loss of generality, we may work with the partial flag  $a_\bullet$ given by
$
\emptyset  = [n_0] \subset [n_1] \subset \cdots \subset [n_k] \subset  [n_{k+1}] = [n]
$
with $0 = n_0 <n_1 <\cdots < n_k < n_{k+1} = n$.

The proof will be an inductive application of Lemmas~\ref{l: useful 1} and ~\ref{l: useful 2}. We work throughout in the manifold $M = \CC^n \times D^n$, and with the initial compact subset $X = \{|x|^2 \leq 1, |z|^2 \leq 1\}$. Fix an initial stratification $\frX$ of $X$ so that $\cF|_X$ is locally constant along the strata. Given a subset $X \subset M$
and stratification $\frX$, at each application of Lemma~\ref{l: useful 1} to a function $f:M\to \RR_{\geq 0}$, we will continue the proof working in the subset $X\cap \{f\geq \epsilon\}$ with the induced stratification $\frX \cap  \{f\geq \epsilon\}$. 
Similarly, given a subset $X \subset M$, at each application of Lemma~\ref{l: useful 2} to a function $f:M\to \RR_{\geq 0}$, we will continue the proof working in the subset $X\cap \{f= \epsilon\}$ with the induced stratification $\frX \cap  \{f= \epsilon\}$. 

First, choose $\frz>0$ sufficiently small as in Lemma~\ref{l:trans}.

Step (0) (a)  Apply Lemma~\ref{l: useful 2} to $r_\emptyset$. 

Step (0) (b)   Consider the poset $\frP(\emptyset, [n_1])$ of all $\emptyset \subsetneq a \subsetneq [n_1]$ with partial order given by inclusion.
Following the partial order of $\frP([0], [n_1])$,  
 apply Lemma~\ref{l: useful 1} sequentially to $r_{[a]}$, for $a\in \frP(\emptyset, [n_1])$.

Step (1) (a)   Apply Lemma~\ref{l: useful 2} to $r_{[n_1]}$. 

Step (1) (b)  Consider the poset $\frP([n_1], [n_2])$ of all $[n_1] \subsetneq a \subsetneq [n_2]$ with partial order given by inclusion.
Following the partial order of $\frP([n_1], [n_2])$,  
 apply Lemma~\ref{l: useful 1} sequentially to $r_{[a]}$, for $a\in \frP([n_1], [n_2])$.

Continue in this way until the following final step.

Step ($n_k$) (a) Apply Lemma~\ref{l: useful 2} to $r_{[n_k]}$. 

Step ($n_k$) (b) Consider the poset $\frP([n_k], [n[)$ of all $[n_k] \subsetneq a \subsetneq [n]$
 with partial order given by inclusion.
Following the partial order of $\frP([n_k], [n])$,  
 apply Lemma~\ref{l: useful 1} sequentially to $r_{[a]}$, for $a\in \frP([n_k], [n])$.

The asserted equivalence follows immediately from the lemmas.  The commutativity is straightforward to verify by  tracing through the constructions.
\end{proof}

By the lemma, we must show restriction along the inclusion $B(\frz) \times P(a^2_\bullet, \frr) \subset B(\frz) \times P(a^1_\bullet, \frr)$ is an equivalence on sections 
\beq\label{eq: restriction of sects}
\xymatrix{
\rho: \Gamma( B(\frz) \times P(a^1_\bullet, \frr), \cF) \ar[r]^-\sim & 
\Gamma(B(\frz) \times P(a^2_\bullet, \frr), \cF) 
}
\eeq

Note there is a natural monotonic family of submanifolds with corners $P_t \subset D^n$, for $t\in [1, 2]$, with the properties: (i)
$P_1 = P(a^1_\bullet, \frr)$, $P_2 = P(a^2_\bullet, \frr)$, and (ii) $\bigcap_{t\in [1,2]} = P_1$, and $\bigcup_{t\in [1,2]} = P_2$. 
We seek to show the sections $\Gamma( B(\frz) \times P_t, \cF)$, for $t\in [0,1]$, are locally constant, or in other words, propagate with respect to $t \in [1,2]$

Consider the outward conormal Lagrangians 
$$\Lambda_{t} = \Lambda_{B(\frz) \times P_t} \subset T^* \CC^m \times D^n
$$ 
For a sheaf  $\cF\in \Sh_\Lambda(\CC \times D^n)$, †by non-characteristic propagation, obstructions to the propagation of sections  with respect to $t \in [1,2]$ can only occur when 
\beq\label{eq: event}
\Lambda^\oo_{P_t}  \cap \Lambda^\oo \not = \emptyset
\eeq

Recall we assume  $\Lambda$ is non-characteristic for the projection $f:\CC^m \times D^n \to D^n$, in the sense that it contains no covectors of the form $(x, 0; z, \xi_z) \in T^*\CC^m \times T^*D^n$. Thus \eqref{eq: event}  can only happen at covectors of the form $\lambda = (x, \xi_x; z, \xi_z)$ with $\xi_x \not = 0$, or in other words,  covectors that project under 
$$
\xymatrix{
\Pi:  T^*\CC^m \times T^*D^n \to  T^*\CC^m \times D^n = T^*_f & \Pi(x, \xi_x; z, \xi_z) = (x, \xi_x; z) 
}
$$
to non-zero covectors   in the relative cotangent bundle. 
Moreover, 
such covectors $\lambda = (x, \xi_x; z, \xi_z)$ with $\xi_x \not = 0$ only appear  
in $\Lambda_{P_t}$ along the ``vertical boundary" $ S(\frz) \times P_t$ where $|x|^2 = \frz^2$. 

Thus if we never encounter  a codirection $\lambda^\oo \in \Lambda^\oo_{P_t}  \cap \Lambda^\oo$  represented by a covector
$ (x, \xi_x; z, \xi_z) \in T^*\CC^m \times T^*D^n$ with $\xi_x \not = 0$, then \eqref{eq: restriction of sects} is an equivalence by non-characteristic propagation, and we are done. Otherwise, let us write $\lambda^\oo(\frz, \frr)  \in T^\oo_f$ for  the codirection through the image $\Pi (x, \xi_x; z, \xi_z) = (x, \xi_x; z) \in T^*_f$. Note that $\lambda^\oo(\frz, \frr)$ points in the radial codirection $d|x|^2 \subset T^*_x \CC^m$ at its base point $x\in \CC^m$.

 Now suppose we   fix $\frz> 0 $, but continue to encounter such obstructions $\lambda^\oo(\frz, \frr)$ as we take $\frr\to 0$, in particular $r_\emptyset \to 0$. Then, after  possibly passing to a subsequence, we obtain a limit   in the closure of the $f$-projection
 $$
 \xymatrix{
 \lambda^\oo(\frz) \subset \ol\Lambda_f|_{z = 0}
 }
 $$
 Note that $\lambda^\oo(\frz)$ points in the radial codirection $d|x|^2 \subset T^*_x \CC^m$ at its base point $x\in \CC^m$.
 
Finally, suppose we continue to encounter such obstructions as we then take $\frz\to 0$.  
By curve selection, we obtain a curve of radial codirections $\lambda^\oo(\frz) \subset \ol\Lambda_\pi |_{z=0}$, for $\frz \in (0, \epsilon)$.
Such a curve of radial codirections can not have base points tending to the origin and also be isotropic. Thus we conclude 
$ \ol\Lambda_f|_{z = 0}$ is not isotropic. This contradicts the $f$-Thom assumption, and so we could not have encountered the assumed obstructions.
\end{proof}

%
%

Finally, let us state without going into detailed arguments a natural generalization of Theorem~\ref{thm: main} which follows immediately from its repeated application.

For each $a\subset [n]$,  set $f(a) = f|_{X(a)}: X(a) \to  D(a)$ to be the restriction. 
 For each $a' \subset a \subset [n]$,
 let $\pi(a', a): D(a) \to D(a  \setminus a')$ denote the natural projection. 
 Set $g(a', a) = \pi(a', a)\circ f(a) : X(a) \to D(a) \to  D(a\setminus a')$ to be the composition  so that $X(a')= g(a', a)^{-1}(0) \subset X(a)$.

Suppose for each $a\subset [n]$, we have a closed conic Lagrangian $\Lambda(a) \subset T^* X^\times(a)$ that  is $f(a)$-non-characteristic. Suppose   for each $a' \subset a \subset [n]$,  the closure of the $g(a', a)$-projection of $\Lambda(a)$
  satisfies the  compatibility
  $$
\xymatrix{
  \ol{\Lambda(a)}_{g(a', a)}|_{X^\times(a')} \subset \Lambda(a') 
}  $$
It is not difficult to check (see~\cite{NS}) this implies the nearby cycles  $\psi^{a}_{a'}$ preserves such singular support subsets. Moreover,  following Remark~\ref{rem:unclosed}, it also implies $\Lambda(a)$ is $g(a', a)$-Thom at the origin $0\in  D(a\setminus a')$.

 Let $\Subsets([n])$ denote the category with objects subsets $a\subset [n]$ and morphisms $a\to a'$ for  inclusions $a' \subset a$.  Note if we take the nerve of $\Subsets([n])$, we recover $\Flags([n])$ as the simplices of the nerve beginning at $[n]$ and ending at $\emptyset$.
  
  By repeated applications of Theorem~\ref{thm: main}, one can obtain the following. 
\begin{theorem}\label{thm: main univ}
With the above setup, the assignment of  sheaves and nearby cycles provides a map of $\oo$-categories
$$
\xymatrix{
\Subsets([n])  \ar[r] & \dgCat_k
&
a  \ar@{|->}[r] & \Sh_{\Lambda(a)}(X^\times(a))
&
(a' \subset a) \ar@{|->}[r] &  \psi^{a}_{a'}
}
$$ 
\end{theorem}

\begin{remark}
Restricting to chains of morphisms beginning at $[n]$ and ending at $\emptyset$,  one recovers
Theorem~\ref{thm: main} from Theorem~\ref{thm: main univ}.
\end{remark}

\subsection{Real version} For our intended application in~\cite{NYverlinde}, we record here a  version of our  main result in a real rather than complex setting. The arguments can be repeated verbatim simply understanding coordinates to be real rather than complex.

To state it, let us return to the constructions and notation
of Section~\ref{ss:hd base} but with the change that $D= (-1, 1)$ will now denote a real interval, and
$D^\times = D\setminus \{0\}$ the punctured real interval. We assume given a submersion $f: X\to D^n$ of real manifolds, and work exclusively with the naive nearby cycles 
 $$
\xymatrix{
\nu^{a \cup b}_b:\Sh(X^\times(a \cup b )) \ar[r] & \Sh(X^\times (b) ) & \nu^{a \cup b}_b = i_X(a, b)^*j_{X}(a, b)_*
}
$$
since there is no monodromy to consider. Then as in Proposition~\ref{prop: lax},
iterated naive nearby cycles gives a map of $\oo$-categories
$$
\xymatrix{
\psi:\Flags([n])  \ar[r] & \Fun_{\dgCat_k}(\Sh(X^\times), \Sh(X_0 ))
}
$$ 
Note one typically applies the above to sheaves supported on the totally positive quadrant $X^{>0} = f^{-1}((0, 1)^n) \subset X^\times = f^{-1}(  (D^\times)^n)$.

We can repeat the definitions of Section~\ref{s:hypo} to say whether a
 closed conic  Lagrangian $\Lambda \subset T^*X^\times$ is  $f$-non-characteristic and $f$-Thom at the origin $0\in D^n$. 
 
 Note the proof of  Theorem~\ref{thm: main} first separates out the monodromy then is primarily devoted to the naive nearby cycles. In the real setting, we can dispense with the arguments  about the monodromy, and simply repeat verbatim the arguments about the naive nearby cycles. In this way we  obtain the following real version.

\begin{theorem}\label{thm: main real}

Let $\Lambda \subset T^*X^\times$ be a closed conic Lagrangian, and $f :X\to D^n$ a submersion.

Suppose  $\Lambda$ is (i) $f$-non-characteristic and (ii) $f$-Thom at the origin $0\in D^n$

Let $\Sh_\Lambda(X^\times) \subset \Sh(X^\times)$ be the full subcategory of complexes with singular support contained in $\Lambda$.

Then the restricted map of $\oo$-categories
$$
\xymatrix{
\nu:\Flags([n])  \ar[r] & \Fun_{\dgCat_k}(\Sh_\Lambda(X^\times), \Sh(X_0 ))
}
$$ 
lands in the full sub-groupoid of functors and invertible natural transformations. Concretely,
for any $\cF\in \Sh_\Lambda(X^\times)$, and
each pair of partial flags $a^1_\bullet \leq  a^2_\bullet$ in $[n]$, the canonical map of
iterated nearby cycles functors
$$
\xymatrix{
\psi(a^1_\bullet)(\cF)\ar[r] & \psi(a^2_\bullet)(\cF)
}
$$
is an equivalence.
\end{theorem}

\end{document}